\documentclass[a4paper,twoside,10pt]{article}

\newcommand{\mytitle}{Reconstructing the Optical Parameters of a Layered Medium with Optical Coherence Elastography}
\title{\mytitle}

\date{}
\author{Peter Elbau$^1$\\{\footnotesize\href{mailto:peter.elbau@univie.ac.at}{peter.elbau@univie.ac.at}}
\and Leonidas Mindrinos$^1$\\{\footnotesize\href{mailto:leonidas.mindrinos@univie.ac.at}{leonidas.mindrinos@univie.ac.at}}
\and Leopold Veselka$^1$\\{\footnotesize\href{mailto:lepold.veselka@univie.ac.at}{leopold.veselka@univie.ac.at}}}

\usepackage[utf8]{inputenc}

\usepackage{microtype}

\usepackage{amsmath}

\usepackage{longtable}

\usepackage{amssymb}

\usepackage{mathrsfs}

\usepackage{mathtools}

\usepackage{siunitx}

\usepackage[a4paper,centering,bindingoffset=0cm,marginpar=0cm]{geometry}

\usepackage{titlesec}
\usepackage[font=footnotesize,format=plain,labelfont=sc,textfont=sl,width=0.75\textwidth,labelsep=period]{caption}
\titleformat{\section}{\filcenter\sc\large}{\thesection.\;}{0em}{}
\titleformat{\subsection}[runin]{\bf}{\thesubsection.\;}{0em}{}[.]

\newpagestyle{headers}%
{%
	\headrule
	\sethead%
		[\thepage]%
		[\footnotesize\sc P.~Elbau, L.~Mindrinos, and L.~Veselka]%
		[]%
		{}%
		{\footnotesize\sc{Reconstructing the Optical Parameters with OCE}}%
		{\thepage}
}
\usepackage[backend=bibtex,maxnames=15]{biblatex}
\usepackage[pdftex,colorlinks=true,linkcolor=blue,citecolor=green,urlcolor=blue,bookmarks=true,bookmarksnumbered=true]{hyperref}

\usepackage[hyperref,amsmath,thmmarks]{ntheorem}
\usepackage{aliascnt}
\theoremstyle{break}
\theoremseparator{.}
\theoremheaderfont{\kern-1em\normalfont\bfseries}
\newtheorem{lemma}{Lemma}[section]

\newaliascnt{proposition}{lemma}

\aliascntresetthe{proposition}

\newaliascnt{corollary}{lemma}

\aliascntresetthe{corollary}

\newaliascnt{assumptions}{lemma}

\aliascntresetthe{assumptions}

\newaliascnt{invpro}{lemma}

\aliascntresetthe{invpro}

\theorembodyfont{\normalfont}
\newaliascnt{definition}{lemma}
\newtheorem{definition}[definition]{Definition}
\aliascntresetthe{definition}

\theorembodyfont{\normalfont\itshape}
\theoremseparator{:}
\theoremheaderfont{\kern-1em\normalfont\itshape}
\newaliascnt{problem}{lemma}
\newtheorem{problem}[problem]{Problem}
\aliascntresetthe{problem}

\theorembodyfont{\normalfont}
\newaliascnt{example}{lemma}

\aliascntresetthe{example}

\theorembodyfont{\normalfont}
\newaliascnt{convention}{lemma}

\aliascntresetthe{convention}

\theorembodyfont{\normalfont\itshape}
\theoremseparator{:}
\theoremheaderfont{\kern-1em\normalfont\itshape}
\newaliascnt{remark}{lemma}
\newtheorem{remark}[remark]{Remark}
\aliascntresetthe{remark}

\theoremstyle{nonumberplain}
\theorembodyfont{\normalfont}
\newtheorem{proof}{Proof}

    \usepackage[ntheorem]{empheq}

\usepackage{bbm}
\usepackage{bm}
\newcommand{\cchi}{\bm\chi}

\newcommand{\R}{\mathbbm{R}}
\newcommand{\Z}{\mathbbm{Z}}
\newcommand{\C}{\mathbbm{C}}
\newcommand{\Q}{\mathds Q}

\renewcommand{\H}{\mathds H}

\newcommand{\e}{\mathrm e}
\renewcommand{\i}{\mathrm i}
\renewcommand{\d}{\,\mathrm d}

\let\RE\Re
\let\Re=\undefined
\DeclareMathOperator{\Re}{\RE e}
\let\IM\Im
\let\Im=\undefined
\DeclareMathOperator{\Im}{\IM m}

\DeclareMathOperator{\curl}{curl}
\let\div=\undefined
\DeclareMathOperator{\div}{div}
\DeclareMathOperator{\grad}{grad}
\DeclareMathOperator{\supp}{supp}

\usepackage{dsfont}

\usepackage{caption}
\usepackage{subcaption}

\usepackage[ ]{algorithm2e}

\SetCommentSty{mycommfont}

\begin{document}

\maketitle
\hspace*{2.5em}
\parbox[t]{0.9\textwidth}{\footnotesize
\hspace*{-1ex}$^1$Faculty of Mathematics, University of Vienna, Oskar-Morgenstern-Platz 1, 1090 Vienna, Austria}

\vspace*{2em}

\begin{abstract}
In this work we consider the inverse problem of reconstructing the optical properties of a layered medium from an elastography measurement where optical coherence tomography is used as the imaging method. We hereby model the sample as a linear dielectric medium so that the imaging parameter is given by its electric susceptibility, which is a frequency- and depth-dependent parameter. Additionally to the layered structure (assumed to be valid at least in the small illuminated region), we allow for small scatterers which we consider to be randomly distributed, a situation which seems more realistic compared to purely homogeneous layers. We then show that a unique reconstruction of the susceptibility of the medium (after averaging over the small scatterers) can be achieved from optical coherence tomography measurements for different compression states of the medium.


\medskip
\noindent \textbf{AMS:} 65J22, 65M32, 78A46
\end{abstract}

\section{Introduction}\label{sec1}

Optical Coherence Tomography is an imaging modality producing high resolution images of biological tissues. It measures the magnitude of the back-scattered light of a focused laser illumination from a sample as a function of depth and provides cross-sectional or volumetric data by performing a series of multiple axial scans at different positions. Initially, it used to operate in time where a movable mirror was giving the depth information. Later on, frequency-domain optical coherence tomography was introduced where the detector is replaced by a spectrometer and no mechanical movement is needed. We refer to \cite{Bre06, DreFuj15} for an overview of the physics of the experiment and to \cite{ElbMinSch15} for a mathematical description of the problem.

Only lately, the inverse problems arising in optical coherence tomography have attracted the interest from the mathematical community, see, for example, \cite{AmmRomShi17, ElbMinSch18a, SanAraBarCarCor15, TriCar19}. For many years, the proposed and commonly used reconstruction method was just the inverse Fourier transform. This approach is valid only if the properties of the medium are assumed to be frequency-independent in the spectrum of the light source. However, the less assumptions one takes, the more mathematically interesting but also difficult the problem becomes. 

The main assumption, we want to make is that the medium can be (at least locally in the region where the laser beam illuminates the object) well described by a layered structure. Since there are in real measurement images typically multiple small particles visible inside these layers, we will additionally include small, randomly distributed scatterers into the model and calculate the averaged contribution of these particles to the measured fields.

To obtain a reconstruction of the medium, that is, of its electric susceptibility, we consider an elastography setup where optical coherence tomography is used as the imaging system. This so-called optical coherence elastography is done by recording optical coherence tomography data for different compression states of the medium, see \cite{AmmBreMil15, DreHubKra19, NahBauRou13, SunStaYan11} for some recent works dealing with this interesting problem.

Under the assumption that the sample can be described as a linear elastic medium, we show that these measurements can be used to achieve a unique reconstruction of the electric susceptibility of the layered medium.

The paper is organised as follows: In \autoref{Mod_se} we review the main equations describing mathematically how the data in optical coherence tomography is collected and their relation to the optical properties of the medium. In \autoref{DD_se}, we show that the calculation of the back-scattered field can be decomposed into the corresponding subproblems for the single layers, for which we derive the resulting formulæ in \autoref{Scat_se}. Finally, we derive in \autoref{OCE_se} that from the measurements at different compression states a unique reconstruction of the susceptibility becomes feasible.

\section{Modelling the optical coherence tomography measurement}\label{Mod_se}

We model the sample by a dispersive, isotropic, non-magnetic, linear dielectric medium characterised by its scalar electric susceptibility. To include randomly distributed scatterers in the model, we introduce the susceptibility as a random variable; so let $(\mathcal{X},\mathcal A,P)$ be a probability space and write
\[ 
\chi:\mathcal{X}\times\R\times\R^3\to\R,\;(\sigma,t,x)\mapsto\chi_\sigma(t,x) 
\]
for the electric susceptibility of the medium in the state $\sigma$. To have a causal model, we require that $\chi_\sigma(t,x)=0$ for all $t<0$.

The object (in a certain realisation state $\sigma\in\mathcal{X}$) is then probed with a laser beam, described by an incident electric field $E^{(0)}:\R\times\R^3\to\R^3$.

\begin{definition}\label{Mod_Incident_de}
We call $E^{(0)}:\R\times\R^3\to\R^3$ an incident wave (for a susceptibility $\chi:\R\times\R^3\to\R)$ in the homogeneous background $\chi_0:\R\to\R$ if it is a solution of Maxwell's equations for $\chi_0$, that is,
\[ \Delta E^{(0)}(t,x)=\frac1{c^2}\partial_{tt}D^{(0)}(t,x), \]
where $c$ denotes the speed of light and 
\[ D^{(0)}(t,x)=E^{(0)}(t,x)+\int_\R\chi_0(\tau)E^{(0)}(t-\tau,x)\d\tau,\]
and $E^{(0)}$ does not interact with the inhomogeneity for negative times, meaning that
\begin{equation}\label{Mod_Incident_Disj_eq}
\chi(\tau,x)E^{(0)}(t,x) = 0\text{ for all }\tau\in\R,\;t\in(-\infty,0),\;x\in\Omega,
\end{equation}
with $\Omega=\{x\in\R^3\mid\chi(\cdot,x) \ne \chi_0(\cdot,x)\}$.
\end{definition}

We then measure the resulting electric field $E_\sigma:\R\times\R^3\to\R^3$ induced by the incident field $E^{(0)}$ in the presence of the dielectric medium described by the susceptibility $\chi_\sigma$.

\begin{definition}
Let $\chi:\R\times\R^3\to\R$ be a susceptibility and $E^{(0)}:\R\times\R^3\to\R^3$ be an incident wave for $\chi$. Then, we call $E$ the electric field induced by $E^{(0)}$ in the presence of $\chi$ if $E$ is a solution of the equation system
\begin{alignat}{2}
\curl\curl E(t,x)+\frac1{c^2}\partial_{tt}D(t,x) &= 0, &&\text{ for all }t\in\R,\;x\in\R^3, \label{Mod_Wave_eq} \\
E(t,x)-E^{(0)}(t,x)&=0, &&\text{ for all }t\in(-\infty,0),\;x\in\R^3, \label{Mod_Wave_Initial_eq}
\end{alignat}
with the electric displacement field $D:\R\times\R^3\to\R^3$ being related to the electric field via
\[ D(t,x) = E(t,x)+\int_{\R}\chi(\tau,x)E(t-\tau,x)\d\tau. \]
\end{definition}
\begin{remark}
The fact that $E^{(0)}$ does not interact with the object before time $t=0$, see \eqref{Mod_Incident_Disj_eq}, guarantees that $E^{(0)}$ is a solution of \eqref{Mod_Wave_eq} and thus the initial condition in \eqref{Mod_Wave_Initial_eq} is compatible with \eqref{Mod_Wave_eq}.
\end{remark}

Equation \eqref{Mod_Wave_eq} is more conveniently written in Fourier space, where we use the convention
\[ \mathcal F[f](k) = \frac1{(2\pi)^{\frac n2}}\int_{\R^n}f(x)\e^{-\i\left<k,x\right>}\d x \]
for the Fourier transform of a function $f:\R^n\to\R$. For convenience, we also use the shorter notation
\[ \check F(\omega,x) = \sqrt{2\pi}\,\mathcal F^{-1}[t\mapsto F(t,x)](\omega) = \int_{\R} F(t,x)\e^{\i\omega t}\d t \]
for this rescaled inverse Fourier transformation of a function of the form $F:\R\times\R^m\to\R^n$ with respect to the time variable.

\begin{lemma}\label{Mod_Fourier_th}
Let $\chi:\R\times\R^3\to\R$ be a susceptibility, $E^{(0)}:\R\times\R^3\to\R^3$ be an incident wave for $\chi$, and $E$ be the corresponding electric field. Then, $\check E$ solves (uniquely) the vector Helmholtz equation 
\begin{equation}\label{Mod_Fourier_Helm_eq}
\curl\curl\check E(\omega,x)-\frac{\omega^2}{c^2}(1+\check\chi(\omega,x))\check E(\omega,x) = 0\text{ for all }\omega\in\R,\;x\in\R^3,
\end{equation}
with the constraint
\begin{equation}\label{Mod_Fourier_Hardy_eq}
\check E \in \mathcal H(\check E^{(0)}),
\end{equation}
where $\mathcal H(\check E^{(0)})$ is the space of all functions $F:\R\times\R^3\to\R^3$ so that the map $\omega\mapsto (F-\check E^{(0)}) (\omega,x)$ can be holomorphically extended to the space $\H\times\R^3$, where $\H=\{z\in\C\mid\Im z>0\}$ denotes the upper half complex plane, and the extension fulfils
\[ \sup_{\lambda>0}\int_{\R}|(F-\check E^{(0)})(\omega+\i\lambda,x)|^2\d\omega < \infty \]
for every $x\in\R^3$.
\end{lemma}

\begin{proof}
Equation \eqref{Mod_Fourier_Helm_eq} is obtained directly from the application of the Fourier transform to \eqref{Mod_Wave_eq}. The condition \eqref{Mod_Fourier_Hardy_eq} is according to the Paley--Wiener theorem, see, for example, \cite[Theorem 9.2]{Rud87}, equivalent to the condition \eqref{Mod_Wave_Initial_eq}, stating that $t\mapsto (E-E^{(0)})(t,x)$ has for every $x\in\R^3$ only support in $[0,\infty)$.
\end{proof}

In frequency-domain optical coherence tomography, we detect with a spectrometer at a position $x_0\in\R^3$ outside the medium the intensity of the Fourier components of the superposition of the back-scattered light from the sample and the reference beam, which is the reflection of the incident laser beam from a mirror at some fixed position.

Here, we consider two independent measurements for two different positions of the mirror in order to overcome the problem of phase-less data, see \cite{ElbMinVes19_report}. Thus, we obtain the data
\[ m_{0,\sigma}(\omega) = |\check E_\sigma(\omega,x_0)|\text{ and }m_{i,\sigma}(\omega) = |\check E_\sigma(\omega,x_0)+\check E^{(\mathrm r)}_i(\omega,x_0)|,\;i\in\{1,2\},\]
for the two known reference waves $E^{(\mathrm r)}_1:\R\times\R^3\to\R^3$ and $E^{(\mathrm r)}_2:\R\times\R^3\to\R^3$, which are solutions of Maxwell's equations in the homogeneous background medium (usually well approximated by the vacuum). 

We see that if the points $0$, $ \check E^{(\mathrm r)}_1(\omega,x_0)$, and $\check E^{(\mathrm r)}_2(\omega,x_0)$ in the complex plane do not lie on a single straight line, we can recover the complex valued electric field $\check E(\omega,x_0)$ for every $\omega\in\R$ by intersecting the three circles
\[\partial B_{m_{0,\sigma}(\omega)}(0)\cap\partial B_{m_{1,\sigma}(\omega)}(-\check E^{(\mathrm r)}_1(\omega,x_0))\cap\partial B_{m_{2,\sigma}(\omega)}(-\check E^{(\mathrm r)}_2(\omega,x_0)).
\]
In the following, we assume that the fields $E^{(\mathrm r)}_1$ and $E^{(\mathrm r)}_2$ are chosen such that the above condition is satisfied and we can recover the function
\[ m_\sigma(\omega) = \check E_\sigma(\omega,x_0)\text{ for all }\omega\in\R. \]

However, this information is still not enough for reconstructing the material parameter $\chi$, see, for example, \cite{ElbMinSch15}. Thus, we make the a priori assumption that the illuminated region of the medium can be well approximated by a layered medium.
Since the layers are typically not completely homogeneous, we also allow for randomly distributed small inclusions in every layer. 
 
Thus, we describe $\chi$ to be of the form
\begin{equation}\label{Mod_susc_eq}
\chi_\sigma(t,x) = \chi_j (t)+\psi_{j,\sigma_j}(t,x)
\end{equation}
in the $j$-th layer $\{x\in\R^3\mid z_{j+1}<x_3<z_j\}$, $j\in\{1,\ldots,J\}$, where we write the measure space as a product $\mathcal X=\prod_{j=1}^J\mathcal X_j$ with each factor representing the state of one layer.
Here, $\chi_j$ is the homogeneous background susceptibility of the layer and $\psi_j$ is the random contribution caused by some small particles in the layer. Outside these layers, we set $\chi_\sigma(t,x)=\chi_0(t)$ for some homogeneous background susceptibility $\chi_0$.

To simplify the analysis, we will assume that the scatterers in the $j$-th layer only occur at some distance to the layer boundaries $z_j$ and $z_{j+1}$, say between $Z_j$ and $\zeta_j$, where $z_{j+1}<Z_j<\zeta_j<z_j$. Moreover, we choose the particles independently, identically, uniformly distributed on the part $U_{j,L_j}=[-\frac12L_j,\frac12L_j]\times[-\frac12L_j,\frac12L_j]\times[Z_j,\zeta_j]$ of the layer for some width $L_j>0$. Concretely, we assume that we have in the $j$-th layer for some number $N_j$ of particles the probability measure $P_{j,N_j,L_j}$ on the probability space $\mathcal X_j=(U_{j,L_j})^{N_j}$ given by
\begin{equation}\label{Mod_Prob_eq}
P_{j,N_j,L_j}({\textstyle\prod_{\ell=1}^{N_j}}A_\ell) = \prod_{\ell=1}^{N_j}\frac{|A_\ell|}{L_j^2(\zeta_j-Z_j)}
\end{equation}
for all measurable subsets $A_\ell\subset U_{j,L_j}$, where $|A_\ell|$ denotes the three dimensional Lebesgue measure of the set $A_\ell$.

The full probability measure $P=P_{N,L}$ is consistently chosen as the direct product $P_{N,L}=\prod_{j=1}^JP_{j,N_j,L_j}$ on $\mathcal X=\prod_{j=1}^J\mathcal X_j$.

The particles themselves, we model in each layer as identical balls with a sufficiently small radius $R$ and a homogeneous susceptibility $\chi^{(\mathrm p)}_j$. Thus, we define for a realisation $\sigma_j\in\mathcal X_j$ of the $j$-th layer the contribution of the particles to the susceptibility by
\begin{equation}\label{Mod_suscRandom_eq}
\psi_{j,\sigma_j}(t,x) = \sum_{\ell=1}^{N_j}\cchi_{B_R(\sigma_{j,\ell})}(x)\,(\chi^{(\mathrm p)}_j(t)-\chi_j(t)),
\end{equation}
where we ignore the problem of overlapping particles. We denote by $\cchi_A$ the characteristic function of a set $A$ and by $B_r(y)$ the open ball with radius $r$ around a point~$y$.

\section{Domain decomposition of the solution}\label{DD_se}

The layered structure of the medium allows us to decompose the solution as a series of solution operators for the single layers. To do so, we split the medium at a horizontal stripe where the medium is homogeneous and consider the two subproblems where once the region above and once the region below is replaced by the homogeneous susceptibility $X_0:\R\to\R$ in the stripe. We write the stripe as the set $\{x\in\R^3\mid z-\varepsilon<x_3<z+\varepsilon\}$ for some $z\in\R$ and some height $\varepsilon>0$ and parametrise the electric susceptibility in the form
\begin{equation}\label{DD_suscept_eq}
\chi(t,x) = \begin{cases}X_1(t,x)&\text{if }x\in \Omega_1 =\{y\in\R^3\mid y_3>z-\varepsilon\},\\ 
X_2(t,x)&\text{if }x\in \Omega_2=\{y\in\R^3\mid y_3<z +\varepsilon\}.\end{cases}
\end{equation}
with the necessary compatibility condition that $X_1$ and $X_2$ coincide in the intersection $\Omega_1 \cap \Omega_2$, where they should both be equal to the homogeneous susceptibility $X_0$.

Additionally, we have the assumption that the medium is bounded in vertical direction. We can therefore assume that for some $z_-<z_+$, the susceptibilities $X_1$ and $X_2$ are homogeneous in $\Omega_+=\{x\in\R^3\mid x_3>z_+\}\subset \Omega_1$ and $\Omega_-=\{x\in\R^3\mid x_3<z_-\}\subset \Omega_2$, respectively. We set
\[ X_1(t,x)=X_+(t)\text{ for all }x\in \Omega_+\text{ and }X_2(t,x)=X_-(t)\text{ for all }x\in \Omega_-. \] 

Since we are solving Maxwell's equations on the whole space, we extend $X_1$ and $X_2$ by the homogeneous susceptibility $X_0$:
\[ X_1(t,x) = X_0(t) \text{ for all } x\in \Omega_2\text{ and }X_2(t,x) = X_0(t) \text{ for all } x\in \Omega_1, \]
see picture (a) in \autoref{fig1} for an illustration of the notation. 

\begin{figure}[t!]
\centering
\begin{subfigure}[b]{.4\textwidth}
\includegraphics[height=7cm]{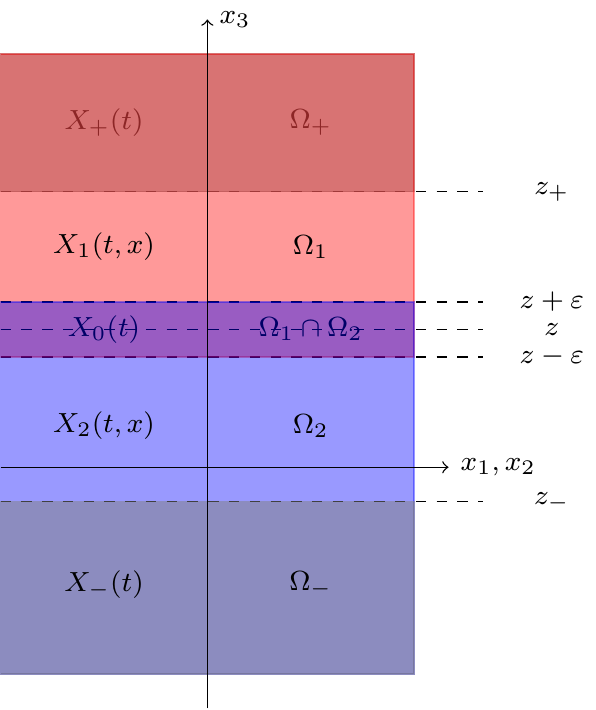}
\caption{The subdomains and the \\
corresponding optical parameters.}
\end{subfigure}\qquad
\begin{subfigure}[b]{.4\textwidth}
\includegraphics[height=3cm]{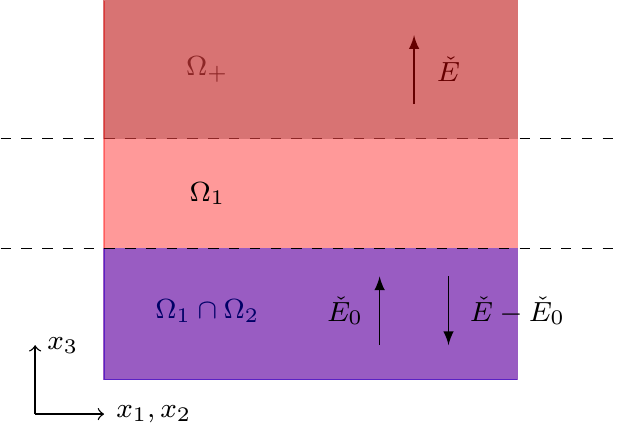}
\caption{The fields related to the operator $\mathcal{G}_1$.}

\vspace{2ex}

\includegraphics[height=3cm]{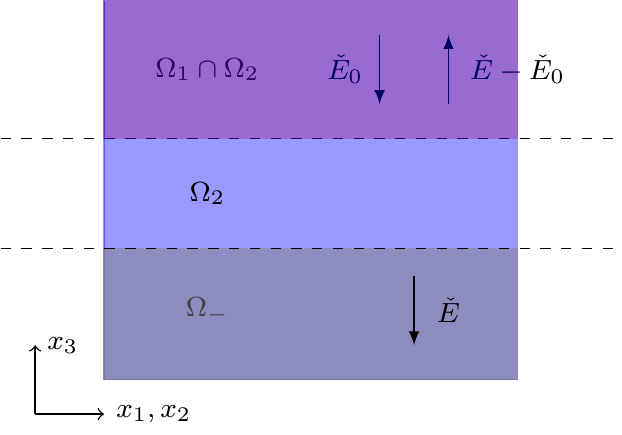}
\caption{The fields related to the operator $\mathcal{G}_2$.}
\end{subfigure}
\caption{The geometry and the notation used in this section.}
\label{fig1}
\end{figure}

The aim is then to reduce the calculation of the electric field in the presence of $\chi$ to the subproblems of determining the electric fields in the presence of $X_1$ and $X_2$, independently. To do so, we consider the solution in the intersection $\Omega_1\cap \Omega_2$ and split it there into waves moving in the positive and negative $e_3$ direction.

\begin{lemma}\label{DD_SolHom_th}
Let a homogeneous susceptibility $\chi:\R\to\R$ be given on a stripe $\Omega_0=\{x\in\R^3\mid x_3\in(z_0-\varepsilon,z_0+\varepsilon)\}$. Then, every solution $\check E:\R\times\R^3\to\C^3$ of
\begin{equation}\label{DD_SolHom_VecHelm_eq}
\curl\curl\check E(\omega,x)-\frac{\omega^2}{c^2}(1+\check\chi(\omega))\check E(\omega,x) = 0\text{ for all }\omega\in\R,\;x\in \Omega_0,
\end{equation}
admits the form
\begin{multline}\label{DD_SolHom_FourForm_eq}
\check E(\omega,x) = \int_{\R^2}e_1(k_1,k_2)\e^{-\i x_3\sqrt{\frac{\omega^2}{c^2}(1+\check\chi(\omega))-k_1^2-k_2^2}}\e^{\i(k_1x_1+k_2x_2)}\d(k_1,k_2) \\
+\int_{\R^2}e_2(k_1,k_2)\e^{\i x_3\sqrt{\frac{\omega^2}{c^2}(1+\check\chi(\omega))-k_1^2-k_2^2}}\e^{\i(k_1x_1+k_2x_2)}\d(k_1,k_2)
\end{multline}
for some coefficients $e_1,\,e_2:\R^2\to\C^3$.
\end{lemma}

\begin{proof}
Taking the divergence of \eqref{DD_SolHom_VecHelm_eq}, we see that $\div\check E=0.$ Then, equation \eqref{DD_SolHom_VecHelm_eq} reduces to the three independent Helmholtz equations
\[ \Delta\check E(\omega,x)+\frac{\omega^2}{c^2}(1+\check\chi(\omega))\check E(\omega,x) = 0\text{ for all }\omega\in\R,\;x\in \Omega_0. \]
Applying the Fourier transform with respect to $x_1$ and $x_2$ and solving the resulting ordinary differential equation in $x_3$ gives us \eqref{DD_SolHom_FourForm_eq}.
\end{proof}

\begin{definition}
Let $\check E$ be a solution of the equation \eqref{DD_SolHom_VecHelm_eq} on some stripe $\Omega_0$, written in the form \eqref{DD_SolHom_FourForm_eq}. We then call $\check E$ a downwards moving solution if $e_2=0$ and an upwards moving solution if $e_1=0$.
\end{definition}

Moreover, we define the solution operators $\mathcal G_1$ and $\mathcal G_2$. To avoid having to define an incident wave on the whole space, we replace the condition \eqref{Mod_Fourier_Hardy_eq} by radiation conditions of the form that we specify the upwards moving part on a stripe below the region and the downwards moving part on a stripe above the region.
\begin{definition}\label{DD_SolOp_de}
Let $\chi$ be given as in \eqref{DD_suscept_eq} and $\check E_0$ be an upwards moving solution in $\Omega_1\cap \Omega_2$. Then, we define $\mathcal G_1\check E_0$ as a solution $\check E$ of the equation
\[ \curl\curl\check E(\omega,x)-\frac{\omega^2}{c^2}(1+\check X_1(\omega,x))\check E(\omega,x) = 0 \]
fulfilling the radiation condition that $\check E-\check E_0$ is a downwards moving solution in $\Omega_1\cap \Omega_2$ and that $\check E$ is an upwards moving solution in $\Omega_+$, see picture (b) in \autoref{fig1}.

Analogously, we define $\mathcal G_2\check E_0$ for a downwards moving solution $\check E_0$ in $\Omega_1\cap \Omega_2$ as a solution $\check E$ of the equation
\[ \curl\curl\check E(\omega,x)-\frac{\omega^2}{c^2}(1+\check X_2(\omega,x))\check E(\omega,x) = 0 \]
fulfilling the radiation condition that $\check E-\check E_0$ is an upwards moving solution in $\Omega_1\cap \Omega_2$ and that $\check E$ is a downwards moving solution in $\Omega_-$, see picture (c) in \autoref{fig1}.
\end{definition}

\begin{remark}\label{DD_Uniq_th}
We do not discuss the uniqueness of these solutions at this point, since we will only need the result for particular, simplified problems where the verification that this gives the desired solution can be done directly. 

Instead we will simply assume that the susceptibilities $\chi$, $X_1$, and $X_2$ are such that the only solution $\check E$ in the presence of this susceptibility for which $\check E$ is upwards moving on $\Omega_+$ and downwards moving on $\Omega_-$ is the trivial solution $\check E=0$, meaning that there is only the trivial solution in the absence of an incident wave.
\end{remark}

\begin{lemma}\label{DD_Decomp_th}
Let $\chi$ be given by \eqref{DD_suscept_eq} and denote by $\mathcal G_1$, $\mathcal G_2$ the solution operators as in \autoref{DD_SolOp_de}.
Let further $E^{(0)}$ be an incident wave on $\chi$ which is moving downwards and $E_1$ be the induced electric fields in the presence of $X_1$.

Then, provided the following series converge, we have that the function $E$ defined by
\begin{align*}
\check E(\omega,x) = \begin{cases}\displaystyle\check E_1(\omega,x)+\sum_{j=0}^\infty\mathcal G_1(\tilde{\mathcal G}_2\tilde{\mathcal G}_1)^j\tilde{\mathcal G}_2\check E_1(\omega,x)&\text{if }x\in \Omega_1, \\
\displaystyle \sum_{j=0}^\infty\mathcal G_2(\tilde{\mathcal G}_1\tilde{\mathcal G}_2)^j\check E_1(\omega,x)&\text{if }x\in \Omega_2,
\end{cases}
\end{align*}
where we set $\tilde{\mathcal G}_i=\mathcal G_i-\mathrm{id}$, $i\in\{1,2\}$, is an electric field in the presence of $\chi$ fulfilling the radiation conditions that $\check E-\check E^{(0)}$ is an upwards moving wave in $\Omega_+$ and $\check E$ is a downwards moving wave in $\Omega_-$.
\end{lemma}

\begin{proof}
First, we remark that the composition of the operators is well defined, since $\check E_1\in\mathcal H(\check E^{(0)})$ is a downwards moving solution in $\Omega_1\cap \Omega_2$, see \autoref{Mod_Fourier_th}, the range of $\tilde{\mathcal G}_2$ consists of upwards moving solutions, and the range of $\tilde{\mathcal G}_1$ consists of downwards moving solutions.

The field $\check E$ is seen to satisfy \eqref{Mod_Fourier_Helm_eq} in $\Omega_1$ by using the definitions of $E_1$ and the solution operator $\mathcal G_1$  on $\Omega_1$. Similarly, using the definition of $\mathcal G_2,$ we get that the function $\check E$ satisfies \eqref{Mod_Fourier_Helm_eq} in $\Omega_2$.

Therefore, it only remains to check that the two formulas coincide in the intersection $\Omega_1\cap \Omega_2$. Using that $\mathcal G_i=\tilde{\mathcal G}_i+\mathrm{id}$, $i\in\{1,2\}$, we find that
\begin{align*}
\check E_1+\sum_{j=0}^\infty\mathcal G_1(\tilde{\mathcal G}_2\tilde{\mathcal G}_1)^j\tilde{\mathcal G}_2\check E_1
&=\check E_1+\sum_{j=0}^\infty\tilde{\mathcal G}_1(\tilde{\mathcal G}_2\tilde{\mathcal G}_1)^j\tilde{\mathcal G}_2\check E_1+\sum_{j=0}^\infty(\tilde{\mathcal G}_2\tilde{\mathcal G}_1)^j\tilde{\mathcal G}_2\check E_1 \\
&= \sum_{j=0}^\infty(\tilde{\mathcal G}_1\tilde{\mathcal G}_2)^j\check E_1+\sum_{j=0}^\infty\tilde{\mathcal G}_2(\tilde{\mathcal G}_1\tilde{\mathcal G}_2)^j\check E_1
= \sum_{j=0}^\infty\mathcal G_2(\tilde{\mathcal G}_1\tilde{\mathcal G}_2)^j\check E_1.
\end{align*}

Moreover, we have that $\check E - \check E_1$ is by construction an upwards moving wave in $\Omega_+$, and therefore so is $\check E-\check E^{(0)}$. Similarly, the wave $\check E$ is a downwards moving wave in $\Omega_-$.
\end{proof}

If we are in a case where our uniqueness assumption mentioned in \autoref{DD_Uniq_th} holds, then \autoref{DD_Decomp_th} allows us to iteratively reduce the problem of determining the electric field in the presence of the susceptibility $\chi_\sigma$, defined in \eqref{Mod_susc_eq}, to problems of simpler susceptibilities. To this end, we could, for example, successively apply the result to values $z\in(\zeta_j,z_j)$ and $z\in(z_{j+1},Z_j)$, $j=1,\ldots,J$, where each successive step is only used to further simplify the operator $\mathcal G_2$ from the previous step. This thus leads to a sort of layer stripping algorithm, see, for example, \cite{ElbMinVes19_report}, where a similar argument was presented.

\section{Wave propagation through a scattering layer}\label{Scat_se}

Using the above analysis, we can calculate the electric field in the presence of a layered medium of the form \eqref{Mod_susc_eq} as a combination of the solutions of the following two subproblems.

\begin{problem}\label{problem1} Let $j\in\{0,\ldots,J-1\}$. Find the electric field induced by some incident field in the presence of the piecewise homogeneous susceptibility $\chi$ given by
\begin{equation}\label{Scat_SuscLayer_eq}
\chi(t,x)=\begin{cases}\chi_j (t)&\text{if }x_3>z_{j+1},\\ \chi_{j+1} (t)&\text{if }x_3<z_{j+1}.\end{cases}
\end{equation}
\end{problem}

\begin{problem}\label{problem2} Let $\sigma\in\mathcal{X}$ and $j\in\{1,\ldots,J\}$. Find the electric field induced by some incident field in the presence of the susceptibility $\chi$ given by
\begin{equation}\label{Scat_SuscRandom_eq}
\chi(t,x)=\chi_j(t)+\psi_{j,\sigma}(t,x),
\end{equation}
where the function $\psi_j$ is described by \eqref{Mod_suscRandom_eq}.
\end{problem}

We thus fix a layer $j\in\{0,\ldots,J\}$, and to simplify the calculations, we restrict ourselves in both subproblems to an illumination by a downwards moving plane wave of the form 
\begin{equation}\label{Scat_Incident_eq}
\check E^{(0)}(\omega,x) = \check f(\omega)\e^{-\i\frac\omega cn_j (\omega) x_3}\eta
\end{equation}
for some function $f:\R\to\R$ and a polarisation vector $\eta\in\mathds S^1\times\{0\}$. Here we define the complex-valued refractive indices for all $j\in\{0,\ldots,J\}$ by
\begin{equation}\label{Scat_Refractive_eq}
n_j:\R\to\mathds H,\;n_j (\omega) = \sqrt{1+\check\chi_j (\omega)}.
\end{equation}

Then, the solution of \autoref{problem1} can be explicitly written down.

\begin{lemma}\label{Scat_Layered_th}
Let $j\in\{0,\ldots,J-1\}$ and $E^{(0)}$ be the incident wave given in \eqref{Scat_Incident_eq}. Then, the electric field $E$ induced by $E^{(0)}$ in the presence of a susceptibility $\chi$ of the form \eqref{Scat_SuscLayer_eq} is given by
\[ \check E(\omega,x) = \check f(\omega)\left(\e^{-\i\frac\omega cn_j(\omega)x_3}-\frac{n_{j+1}(\omega)-n_j (\omega)}{n_{j+1}(\omega)+n_j (\omega)}\e^{-\i\frac\omega c n_j (\omega)z_{j+1}}\e^{\i\frac\omega c n_j (\omega)(x_3-z_{j+1})}\right)\eta \]
for $x_3>z_{j+1}$, and by
\[
\check E(\omega,x) = \check f(\omega)\frac{2n_j (\omega)}{n_{j+1} (\omega)+n_j (\omega)}\e^{-\i\frac\omega c n_j (\omega)z_{j+1}}\e^{-\i\frac\omega c n_{j+1}(\omega)(x_3-z_{j+1})}\eta
\]
for $x_3<z_{j+1}$, where the refractive indices $n_j$ and $n_{j+1}$ are defined by \eqref{Scat_Refractive_eq}.
\end{lemma}

\begin{proof}
Clearly, $\check E$ satisfies the differential equation \eqref{Mod_Fourier_Helm_eq} in both regions $x_3>z_{j+1}$ and $x_3<z_{j+1}$. Moreover, $\check E^{(0)}$ is the only incoming wave in $\check E$. Therefore, it only remains to check that $\check E$ has sufficient regularity to be the weak solution along the discontinuity of the susceptibility at $x_3=z_{j+1}$, meaning that
\begin{align*}
\lim_{x_3\uparrow z_{j+1}}\check E(\omega,x) &= \lim_{x_3\downarrow z_{j+1}}\check E(\omega,x), \\
\lim_{x_3\uparrow z_{j+1}}n_{j+1}(\omega)\partial_{x_3}\check E(\omega,x) &= \lim_{x_3\downarrow z_{j+1}}n_j (\omega)\partial_{x_3}\check E(\omega,x).
\end{align*}
Both identities are readily verified.
\end{proof}

For \autoref{problem2}, the situation is more complicated and we settle for an approximate solution for the electric field. For that, we assume (using the same notation as in \eqref{Mod_suscRandom_eq}) that the susceptibility $\chi^{(\mathrm p)}_j$ of the random particles does not differ much from the background $\chi_j$, so that the difference between the induced field and the incident field becomes small, and we do a first order approximation in the difference $\chi^{(\mathrm p)}_j-\chi_j$.
For that purpose, we write the differential equation \eqref{Mod_Fourier_Helm_eq} in the form
\[ \curl\curl\check E(\omega,x)-\frac{\omega^2}{c^2}n_j^2(\omega)(1+\bar\phi_{j,\sigma_j}(\omega,x))\check E(\omega,x) = 0, \]
where, according to \eqref{Mod_suscRandom_eq},
\[ \bar\phi_{j,\sigma_j}(\omega,x) = \sum_{\ell=1}^{N_j}\cchi_{B_R(\sigma_{j,\ell})}(x)\,\phi_j(\omega), \]
and we abbreviate
\begin{equation}\label{Scat_phi_eq}
\phi_j(\omega)=\frac{\check\chi^{(\mathrm p)}_j(\omega,x)-\check\chi_j(\omega)}{1+\check\chi_j(\omega)}.
\end{equation}

In first order in $\bar\phi$, we then approximate the field by the solution $\check E^{(1)}_{N_j,\sigma_j}$ of the equation
\[ \curl\curl\check E^{(1)}(\omega,x)-\frac{\omega^2}{c^2}n_j^2(\omega)\check E^{(1)}(\omega,x) = \bar\phi_{j,\sigma_j}(\omega,x)\check E^{(0)}(\omega,x), \]
the so called Born approximation. Using that the fundamental solution $G$ of the Helmholtz equation, which by definition fulfils
\[ \Delta G(\kappa,x)+\kappa^2G(\kappa,x) = -\delta(x), \]
is given by
\[ G(\kappa,x)=\frac{\e^{\i\kappa|x|}}{4\pi|x|}, \]
we obtain the expression
\begin{multline}\label{Scat_BornSol_eq}
\check E_{N_j,\sigma_j}^{(1)}(\omega,x) = \check E^{(0)}(\omega,x) \\
+\left(\frac{\omega^2}{c^2}n_j^2(\omega)+\grad\div\right)\sum_{\ell=1}^{N_j}\int_{B_R(\sigma_{j,\ell})}G(\tfrac\omega cn_j(\omega),x-y)\phi_j(\omega)\check E^{(0)}(\omega,y)\d y
\end{multline}
for the Born approximation of the induced field, see, for example, \cite[Proposition~4]{ElbMinSch15}.

We now want to determine the expected value of $E^{(1)}_{N_j,\sigma_j}$ in the limit where the number of particles $N_j$ and the width $L_j$ of the region where the particles are horizontally distributed tend to infinity, while keeping the ratio $\rho_j=\frac{N_j}{L_j^2}$ of particles per surface area constant, that is, we want to calculate the expression
\begin{equation}\label{Scat_RandAver_eq}
\bar E^{(1)}(\omega,x) = \lim_{N_j\to\infty}\int_{\mathcal X_j}\check E_{N_j,\sigma_j}^{(1)}(\omega,x)\d P_{j,N_j,L_j(N_j)}(\sigma_j),
\end{equation}
where $L_j(N_j)=\sqrt{\tfrac{N_j}{\rho_j}}$ and $P$ denotes the probability measure introduced in \eqref{Mod_Prob_eq}.

\begin{lemma}\label{Scat_Random_th}
Let $j\in\{1,\ldots,J\}$ and $\rho_j>0$ be fixed, $E^{(0)}$ be an incident field of the form \eqref{Scat_Incident_eq}, and $\chi$ be the susceptibility specified in \eqref{Scat_SuscRandom_eq}.

Then, the expected value $\bar E^{(1)}$ of the Born approximation of the field induced by~$E^{(0)}$ in the presence of the susceptibility $\chi$ in the limit $N_j\to\infty$ with $L_j^2\rho_j=N_j$, as introduced in \eqref{Scat_RandAver_eq}, is given by
\begin{multline}\label{Scat_Random_Refl_eq}
\bar E^{(1)}(\omega,x) = \check E^{(0)}(\omega,x)+(2\pi)^4\rho_j\phi_j(\omega)\check f(\omega) \\
\times h(2R\tfrac\omega cn_j(\omega))\left(\e^{-\i\frac\omega cn_j(\omega)Z_j}-\e^{-\i\frac\omega cn_j(\omega)\zeta_j}\right)\e^{\i\frac\omega cn_j(\omega)(x_3-\mu_j)}\eta
\end{multline}
for $x_3>\zeta_j+R$ and by
\begin{multline}\label{Scat_Random_Trans_eq}
\bar E^{(1)}(\omega,x) = \check E^{(0)}(\omega,x)+\frac{(2\pi)^4}3\rho_j\phi_j(\omega)\check f(\omega) \\
\times\left(\e^{-\i\frac\omega cn_j(\omega)Z_j}-\e^{-\i\frac\omega cn_j(\omega)\zeta_j}\right)\e^{-\i\frac\omega cn_j(\omega)(x_3-\mu_j)}\eta
\end{multline}
for $x_3<Z_j-R$, where $\mu_j=\frac12(\zeta_j+Z_j)$ and
\begin{equation}\label{Scat_Random_h_eq}
h(\xi) = \frac{\sin(\xi)-\xi\cos(\xi)}{\xi^3}.
\end{equation}
\end{lemma}

\begin{proof}
Inserting the expression \eqref{Scat_BornSol_eq} for the Born approximation of the electric field into the formula \eqref{Scat_RandAver_eq} for the expected value, we obtain the equation
\begin{equation}\label{Scat_Random_Field_eq}
\bar E^{(1)}(\omega,x) = \check E^{(0)}(\omega,x) 
+\lim_{N_j\to\infty}N_j\phi_j(\omega)\check f(\omega)\left(\frac{\omega^2}{c^2}n_j^2(\omega)+\grad\div\right)K_{L_j(N_j)}(\omega,x)\eta,
\end{equation}
where
\[ K_L(\omega,x) = \int_{U_{j,L}}\int_{B_R(\sigma_{j,1})}G(\tfrac\omega cn_j(\omega),x-y)\e^{-\i\frac\omega cn_j (\omega)x_3}\d y\d\sigma_{j,1}. \]
We recall that $U_{j,L}=[-\frac12L,\frac12L]\times[-\frac12L,\frac12L]\times[Z_j,\zeta_j]$ is for $L=L_j$ the region in which the particles in the $j$-th layer are lying. To symmetrise the expression, we set
\[ \mu_j=\frac12(\zeta_j+Z_j)\text{ and }d_j=\frac12(\zeta_j-Z_j) \]
and shift $U_{j,L}$ to the origin, by defining $\tilde U_{j,L}=U_{j,L}-\mu_j e_3$ with $e_3=(0,0,1)$.

Introducing the probability density
\[ p_L(\xi)=\frac1{|U_{j,L}|}\cchi_{U_{j,L}}(\mu_j e_3+\xi)=\frac1{2L^2d_j}\cchi_{\tilde U_{j,L}}(\xi) \]
for the variable $\xi=\sigma_{j,1}-\mu_j e_3$, we rewrite $K_L$ in the form
\begin{align*}
K_L(\omega,x) &= \int_{\R^3}p_L(\xi)\e^{-\i\frac\omega cn_j(\omega)(\mu_j+\xi_3)} \\
&\hspace{0.5cm}\times\int_{\R^3}\cchi_{B_R(0)}(y)G(\tfrac\omega cn_j(\omega),x-\mu_j e_3-\xi-y)\e^{-\i\frac\omega cn_j(\omega) y_3}\d y\d\xi \\
&= (2\pi)^{\frac32}\int_{\R^3}p_L(\xi)\e^{-\i\frac\omega cn_j(\omega)(\mu_j+\xi_3)} \\
&\hspace{0.5cm}\times\mathcal F[y\mapsto\cchi_{B_R(0)}(y)G(\tfrac\omega cn_j(\omega),x-\mu_j e_3-\xi-y)](\tfrac\omega cn_j(\omega)e_3)\d\xi \\
&= (2\pi)^3\int_{\R^3}p_L(\xi)\e^{-\i\frac\omega cn_j(\omega)(\mu_j+\xi_3)} \\
&\hspace{0.5cm}\times\big(\mathcal F[\cchi_{B_R(0)}]*\mathcal F[y\mapsto G(\tfrac\omega cn_j(\omega),x-\mu_j e_3-\xi-y)]\big)(\tfrac\omega cn_j(\omega)e_3)\d\xi.
\end{align*}
Since $G(\kappa,y)=G(\kappa,-y)$, we have with $\hat G(\kappa,k)=\mathcal F[y\mapsto G(\kappa,y)](k)$ that
\[ \mathcal F[y\mapsto G(\kappa,x-\mu_j e_3-\xi-y)](k) = \e^{-\i\left<k,x-\mu_j e_3-\xi\right>}\hat G(\kappa,k). \]
Therefore, we can write this with the notation $\hat\cchi_{B_R(0)}=\mathcal F[\cchi_{B_R(0)}]$ and $\hat p_L=\mathcal F[p_L]$ as
\begin{equation}\label{Scat_Random_KLFirstStep_eq}
\begin{split}
K_L(\omega,x) &= (2\pi)^3\e^{-\i\frac\omega cn_j(\omega)\mu_j}\int_{\R^3}\hat\cchi_{B_R(0)}(\tfrac\omega cn_j(\omega)e_3-k)\hat G(\tfrac\omega cn_j(\omega),k) \\
&\hspace{2cm}\times\e^{-\i\left<k,x-\mu_j e_3\right>}\int_{\R^3}p_L(\xi)\e^{-\i\left<\frac\omega cn_j(\omega)e_3-k),\xi\right>}\d\xi\d k \\
&= (2\pi)^{\frac92}\e^{-\i\frac\omega cn_j(\omega)\mu_j}\int_{\R^3}\hat\cchi_{B_R(0)}(\tfrac\omega cn_j(\omega)e_3-k)\\
&\hspace{2cm}\times\hat p_L(\tfrac\omega cn_j(\omega)e_3-k)\hat G(\tfrac\omega cn_j(\omega),k)\e^{-\i\left<k,x-\mu_j e_3\right>}\d k.
\end{split}
\end{equation}

Remarking that
\begin{align*}
\hat p_L(k) &= \frac1{(2\pi)^{\frac32}L^2}\int_{-\frac L2}^{\frac L2}\e^{-\i k_1\xi_1}\d\xi_1\int_{-\frac L2}^{\frac L2}\e^{-\i k_2\xi_2}\d\xi_2\int_\R\cchi_{[-d_j,d_j]}(\xi_3)\e^{-\i k_3\xi_3}\d\xi_3 \\
&= \frac1{(2\pi)^{\frac32}L^2}\frac{2\sin(\tfrac12Lk_1)}{k_1}\frac{2\sin(\tfrac12Lk_2)}{k_2}\int_\R\cchi_{[-d_j,d_j]}(\xi_3)\e^{-\i k_3\xi_3}\d\xi_3,
\end{align*}
we see that we have with $\hat\cchi_{[-d_j,d_j]}=\mathcal F[\cchi_{[-d_j,d_j]}]$ the limit
\begin{equation}\label{Scat_Random_Delta_eq}
N_j\hat p_{L_j(N_j)}(k)\to2\pi\rho_j\delta(k_1)\delta(k_2)\hat\cchi_{[-d_j,d_j]}(k_3)\;(N_j\to\infty).
\end{equation}

Using \eqref{Scat_Random_Delta_eq} in \eqref{Scat_Random_KLFirstStep_eq}, we can calculate the behaviour of $K_L$ in this limit to be
\begin{multline*}
\lim_{N_j\to\infty}N_jK_{L_j(N_j)}(\omega,x) = (2\pi)^{\frac{11}2}\e^{-\i\frac\omega cn_j(\omega)\mu_j}\rho_j \\
\times\int_\R\hat\cchi_{B_R(0)}((\tfrac\omega cn_j(\omega)-k_3)e_3)\hat G(\tfrac\omega cn_j(\omega),k_3e_3)\hat\cchi_{[-d_j,d_j]}(k_3)\e^{-\i k_3(x_3-\mu_j)}\d k_3.
\end{multline*}

Using further that $\hat G$ can be computed by taking the Fourier transform of the Helmholtz equation, giving us
\[ \hat G(\kappa,k) = \frac1{(2\pi)^{\frac32}}\frac1{|k|^2-\kappa^2}, \]
and calculating the Fourier transform of the characteristic function of a sphere to be
\begin{align*}
\hat\cchi_{B_R(0)}(k) &= \frac1{\sqrt{2\pi}}\int_0^R\int_0^\pi r^2\sin\theta\e^{-\i r|k|\cos\theta}\d\theta\d r 
= \frac1{\sqrt{2\pi}}\int_0^R\frac r{\i |k|}(\e^{\i r|k|}-\e^{-\i r|k|})\d r \\
&= \frac1{|k|^3}\sqrt{\frac2\pi}\int_0^{R|k|}\alpha\sin(\alpha)\d\alpha
= \frac1{|k|^3}\sqrt{\frac2\pi}(\sin(R|k|)-R|k|\cos(R|k|));
\end{align*}
we are left with
\begin{multline}\label{Scat_Random_KLSecondStep_eq}
\lim_{N_j\to\infty}N_jK_{L_j(N_j)}(\omega,x) = (2\pi)^4\e^{-\i\frac\omega cn_j(\omega)\mu_j}\rho_j\sqrt{\frac2\pi} \\
\times\int_{\R}h(R(\tfrac\omega cn_j(\omega)-k_3))\frac1{k_3^2-\frac{\omega^2}{c^2}n_j^2(\omega)}\hat\cchi_{[-d_j,d_j]}(k_3)\e^{-\i k_3(x_3-\mu_j)}\d k_3,
\end{multline}
where we used the abbreviation $h$ from \eqref{Scat_Random_h_eq}.

Inserting finally
\[ \hat\cchi_{[-d_j,d_j]}(k_3) = \frac1{\sqrt{2\pi}}\int_{-d_j}^{d_j}\e^{-\i k_3x_3}\d x_3 = \frac1{\sqrt{2\pi}}\frac1{\i k_3}\left(\e^{\i k_3d_j}-\e^{-\i k_3d_j}\right), \]
we see that the integrand in \eqref{Scat_Random_KLSecondStep_eq} can for $x_3-\mu_j>d_j+R$ (that is, for $x_3>\zeta_j+R$) be meromorphically extended to a function of $k_3$ in the lower half complex plane which decays sufficiently fast at infinity, so that the residue theorem yields
\begin{multline*}
\lim_{N_j\to\infty}N_jK_{L_j(N_j)}(\omega,x) = (2\pi)^4\e^{-\i\frac\omega cn_j(\omega)\mu_j}\rho_j\frac{h(2R\tfrac\omega cn_j(\omega))}{\frac{\omega^2}{c^2}n_j^2(\omega)} \\
\times\left(\e^{\i\frac\omega cn_j(\omega)d_j}-\e^{-\i\frac\omega cn_j(\omega)d_j}\right)\e^{\i\frac\omega cn_j(\omega)(x_3-\mu_j)}.
\end{multline*}
Putting this into \eqref{Scat_Random_Field_eq}, we obtain with $\mu_j+d_j=\zeta_j$ and $\mu_j-d_j=Z_j$ the formula \eqref{Scat_Random_Refl_eq}.

Similarly, we extend the integrand for $x_3-\mu_j<-d_j-R$ (that is, for $x_3<Z_j-R$) meromorphically to a function of $k_3$ in the upper half plane and find with the residue theorem that
\[
\lim_{N_j\to\infty}N_jK_{L_j(N_j)}(\omega,x) = (2\pi)^4\rho_j\frac{h(0)}{\frac{\omega^2}{c^2}n_j^2(\omega)}
\times\left(\e^{\i\frac\omega cn_j(\omega)d_j}-\e^{-\i\frac\omega cn_j(\omega)d_j}\right)\e^{-\i\frac\omega cn_j(\omega)x_3},
\]
which gives us with \eqref{Scat_Random_Field_eq} and with $h(0)=\frac13$ the formula \eqref{Scat_Random_Trans_eq}.
\end{proof}

\section{Recovering the susceptibility with optical coherence elastography}\label{OCE_se}

So far, we have presented a way to model the measurements of an optical coherence tomography setup for a layered medium of the form \eqref{Mod_susc_eq}. The question we are really interested in, however, is how to reconstruct the properties of the medium from this data.

Let us first consider one of the layer stripping steps for a susceptibility $\chi$ of the form \eqref{DD_suscept_eq} with $X_1$ being either of the form \eqref{Scat_SuscLayer_eq} of \autoref{problem1} or of the form~\eqref{Scat_SuscRandom_eq} of \autoref{problem2}. We make the additional assumption that $\supp\chi_j\subset[0,T]$ and $\supp\chi^{(\mathrm p)}_j\subset[0,T]$ for a sufficiently small $T>0$. Then, we see that by choosing a sufficiently short pulse as incident wave, that is, $E^{(0)}(t,x)=f(t+\frac{x_3}c)\eta$ (assuming for the background medium $\chi_0=0$) with $f$ having a sufficiently narrow support (this ability is of course limited by the available frequencies), we can arrange it such that the field $E$ in the presence of $\chi$ and the field $E_1$ in the presence of $X_1$ are such that $E_1(t,x_0)=E(t,x_0)$ for all $t<t_0$ and $E_1(t,x_0)=0$ for $t\ge t_0$ at the detector $x_0\in\R^3$ for some time $t_0\in\R$.
This allows us to split the reconstruction of the electric susceptibility by a layer stripping method and reconstruct each layer separately.

We will therefore only describe the inductive steps, in which we independently consider the subproblems described in \autoref{Scat_se}.

We want to start with measurements from an optical coherence elastography setup, that is, we have optical coherence tomography data for different elastic states of the medium. Concretely, we apply a force proportional to some parameter $\delta\in\R$ perpendicular to the layers of the medium, which causes under the assumption of a linear elastic medium a linear displacement of the position $z_j$ of the layer. Correspondingly, the refractive indices in the medium, defined by \eqref{Scat_Refractive_eq}, will change, which we assume to be linear as well. Thus, each layer at the compression state corresponding to $\delta$ will be characterised by a refractive index $\bar n_j$ and a vertical position~$\bar z_j$ of the beginning of the layer of the form
\[ \bar n_j(\omega,\delta)=n_j(\omega)+\delta n_j'(\omega),\text{ and } \bar z_j(\delta)=z_j+\delta z_j', \]
for some functions $n_j':\R\to\C,$ and a slope $z_j'\in\R$.

In the first reconstruction step, we have that the first layer is the background in which the medium resides, which we assume to be well described by the vacuum $n_0=1$ and not to be affected by the compression, that is, $n_0'=0$. Moreover, the distance between the detector and the medium shall be kept fixed during the compression so that $z_1'=0$ as well.

According to \autoref{Scat_Layered_th}, the measurements at the detector $x_0\in\R^3$ with $x_{0,3}>z_1$ then allow us to extract (knowing $\bar n_0=1$, the incident field $E^{(0)}$, and the position $x_3$ of the detector explicitly) the information
\begin{equation}\label{eq_joint_measLayer0}
m_0[n_1,n_1',z](\omega,\delta) = \frac{\bar n_1(\omega, \delta) -1}{\bar n_1(\omega,\delta)+1} e^{-2\i \tfrac{\omega}c z_1}.
\end{equation}
From this data, we can uniquely compute the functions $n_1$, $n_1'$, and $z_1$.

\begin{lemma}\label{th_joint_uniqRecLayer0}
Let $I\subset\R$ be a set which contains at least two incommensurable points $\omega_1,\omega_2\in I\setminus\{0\}$ (that is, $\frac{\omega_1}{\omega_2}\in\R\setminus\Q$). Assume that we have $(n_1,n_1',z_1)$ and $(\tilde n_1,\tilde n_1',\tilde z_1)$ with $n_1'(\omega)\ne0$, $\tilde n_1'(\omega)\ne0$, and
\begin{equation}\label{eq_joint_uniqRecLayer0_meas}
m_0[n_1,n_1',z_1](\omega,\delta) = m_0[\tilde n_1,\tilde n_1',\tilde z_1](\omega,\delta)\text{ for all }\omega\in I,\;\delta\in\R. 
\end{equation}

Then, we have
\[ n_1(\omega)=\tilde n_1(\omega),\;n_1'(\omega)=\tilde n_1'(\omega),\text{ and } z_1=\tilde z_1 \text{ for all }\omega\in I. \]
\end{lemma}
\begin{proof}
Expanding the fractions in \eqref{eq_joint_uniqRecLayer0_meas}, the equation reduces to the zeroes of a quadratic polynomial in $\delta$. Comparing the coefficients of second order of $\delta$, we find that
\[ n_1'(\omega)\tilde n_1'(\omega)\left(e^{-2\i\tfrac{\omega}cz_1}-e^{-2\i \tfrac{\omega}c\tilde z_1}\right) = 0. \]
Thus, we get
\[ e^{-2\i\frac\omega cz_1}=e^{-2\i\frac\omega c\tilde z_1}\text{ for all }\omega\in I. \]
Evaluating this at $\omega_1$ and $\omega_2$, we have that there exist two integers $\lambda_1,\lambda_2\in\Z$ with
\[ z_1-\tilde z_1=\frac{\pi c}{\omega_1}\lambda_1=\frac{\pi c}{\omega_2}\lambda_2. \]
If $\lambda_2\ne0$, then we would get the contradiction $\frac{\lambda_1}{\lambda_2}=\frac{\omega_1}{\omega_2}\in\R\setminus\Q$. Therefore, $\lambda_2=0,$ which means that $z_1=\tilde z_1$.

With this, \eqref{eq_joint_uniqRecLayer0_meas} evaluated at $\delta=0$ simplifies to
\[ n_1(\omega)=\tilde n_1(\omega)\text{ for all }\omega\in I. \]

Finally, looking at the terms of first order in $\delta$ in the expanded version of \eqref{eq_joint_uniqRecLayer0_meas}, we find that they have been reduced to give the equation
\[ n_1'(\omega)=\tilde n_1'(\omega). \]
\end{proof}

After having recovered the parameters up to the $j$-th layer, $j\in\{1,\ldots,J\}$, we can clean our measurement data from all effects caused by the previous layers and consider the next subproblem, namely the signal originating from the region of the randomly distributed particles. Here, the unknown parameters consist of
\begin{itemize}
\item
the radius $R$ of the particles, which we will assume to be so small that the approximation $R=0$ is reasonable and that the particles can also after compression be considered to have a round shape;
\item
the ratio $\rho_j>0$ of particles per surface area, which we assume to be invariant under the compression;
\item
the refractive index $\bar\nu_j$ of the particles, which we assume to deform linearly according to
\[ \bar\nu_j(\omega,\delta) = \nu_j(\omega)+\delta\nu_j'(\omega),\text{ where }\nu_j(\omega)=\sqrt{1+\check\chi^{(\mathrm p)}_j(\omega)}, \]
under compression; and
\item
the vertical positions $\bar\zeta_j$ and $\bar Z_j$ of the beginning and the end of the random medium inside the $j$-th layer, which are also assumed to change linearly according to
\[ \bar\zeta_j(\delta) = \zeta_j+\delta\zeta_j'\text{ and }\bar Z_j(\delta) = Z_j+\delta Z_j'. \]
\end{itemize}
We collect these unknowns in the tuple $S_j=(\rho_j,\nu_j,\nu_j',\zeta_j,\zeta_j',Z_j,Z_j')$.
The (corrected) incident wave $E^{(0)}$ and the refractive index $n_j$ and its rate $n_j'$ of change under compression are presumed to be already calculated.

From the measurements of the electric field for this subproblem, provided that it can be well approximated by the expected value of the Born approximation as calculated in \autoref{Scat_Random_th}, we can extract the data (rewriting the expression \eqref{Scat_phi_eq} for $\phi_j$ in \eqref{Scat_Random_Refl_eq} in terms of the refractive indices)
\[
M_j[S_j](\omega,\delta) = \rho_j(\bar \nu_j^2(\omega,\delta)-\bar n_j^2(\omega,\delta)) \left(\e^{-\i\frac\omega{2c}\bar n_j(\omega,\delta)(\bar\zeta_j(\delta)+3\bar Z_j(\delta))}-\e^{-\i\frac\omega{2c}\bar n_j(\omega,\delta)(3\bar\zeta_j(\delta)+\bar Z_j(\delta))}\right),
\]

\begin{lemma}\label{OCE_Random_th}
Let $j\in\{1,\ldots,J\}$ be fixed, $I\subset\R$ be an arbitrary subset and $n_j$, $n_j'$ be given such that $n_j(\omega)\ne0$ for every $\omega\in I$ and that there exists a value $\omega_0\in I\setminus\{0\}$ with $\Im(n_j'(\omega_0))>0$. Assume that we have $S_j=(\rho_j,\nu_j,\nu_j',\zeta_j,\zeta_j',Z_j,Z_j')$ and $\tilde S_j=(\tilde\rho_j,\tilde \nu_j,\tilde \nu_j',\tilde \zeta_j,\tilde \zeta_j',\tilde Z_j,\tilde Z_j')$ with
\begin{equation}\label{OCE_Random_Meas_eq}
M_j[S_j](\omega,\delta) = M_j[\tilde S_j](\omega,\delta)\text{ for all }\omega\in I,\;\delta\in\R. 
\end{equation}
Additionally, we enforce the ordering $Z_j<\zeta_j$ and $\tilde Z_j<\tilde\zeta_j$ about the beginning and the end of the random layer and make the assumptions $Z_j'>\zeta_j'>0$ and $\tilde Z_j'>\tilde\zeta_j'>0$ that the layer shrinks when being compressed.

Moreover, we assume the existence of an element $\omega_1\in I$ so that
\begin{equation}\label{OCE_Random_Assum_eq}
\frac{n_j'(\omega_1)}{n_j(\omega_1)}\ne\frac{\nu_j'(\omega_1)}{\nu_j(\omega_1)}.
\end{equation}

Then, we have
\[ S_j=\tilde S_j. \]
\end{lemma}
\begin{proof}
Considering the different orders of decay in $\delta$ in the exponents in \eqref{OCE_Random_Meas_eq}, we require that all of them match, which yields the equation system
\begin{align*}
\delta^2\frac\omega{2c}\Im(n_j'(\omega))(\zeta_j'+3Z_j') &= \delta^2\frac\omega c\Im(n_j'(\omega))(\tilde\zeta_j'+3\tilde Z_j')\text{ and} \\
\delta^2\frac\omega{2c}\Im(n_j'(\omega))(3\zeta_j'+Z_j') &= \delta^2\frac\omega c\Im(n_j'(\omega))(3\tilde\zeta_j'+\tilde Z_j')
\end{align*}
for the exponents quadratic in $\delta$, which implies $\zeta_j'=\tilde\zeta_j'$ and $Z_j'=\tilde Z_j'$, and, using this result, the equation system
\begin{align*}
\delta\frac\omega{2c}\Im(n_j'(\omega))(3\zeta_j+Z_j) &= \delta\frac\omega{2c}\Im(n_j'(\omega))(3\tilde\zeta_j+\tilde Z_j)\text{ and} \\
\delta\frac\omega{2c}\Im(n_j'(\omega))(\zeta_j+3Z_j) &= \delta\frac\omega{2c}\Im(n_j'(\omega))(\tilde\zeta_j+3\tilde Z_j)
\end{align*}
for the exponents linear in $\delta$, which further implies $\zeta_j=\tilde\zeta_j$ and $Z_j=\tilde Z_j$.

At this point, \eqref{OCE_Random_Meas_eq} is reduced to
\[ \rho_j\left((\nu_j+\delta\nu_j')^2-(n_j+\delta n_j')^2\right) = \tilde\rho_j\left((\tilde\nu_j+\delta\tilde\nu_j')^2-(n_j+\delta n_j')^2\right). \]
Comparing coefficients with respect to $\delta$ gives us the equation system
\begin{align}
\rho_j\left(\nu_j'{}^2-n_j'{}^2\right) &= \tilde\rho_j\left(\tilde\nu_j'{}^2-n_j'{}^2\right),\label{OCE_Random_first_eq} \\
\rho_j\left(\nu_j\nu_j'-n_jn_j'\right) &= \tilde\rho_j\left(\tilde\nu_j\tilde\nu_j'-n_j\tilde n_j'\right),\label{OCE_Random_second_eq} \\
\rho_j\left(\nu_j^2-n_j^2\right) &= \tilde\rho_j\left(\tilde\nu_j^2-n_j^2\right).\label{OCE_Random_third_eq}
\end{align}

We use equation \eqref{OCE_Random_third_eq} in \eqref{OCE_Random_first_eq} and \eqref{OCE_Random_second_eq} to eliminate of the variables $\rho_j$ and $\tilde\rho_j$, and interpret the result as an equation system for the variables $\tilde\nu_j$ and $\tilde\nu_j'$. Solving these equations then for $\tilde\nu_j'$, gives us
\begin{align*}
(\nu_j^2-n_j^2)\tilde\nu_j'{}^2 &= (\tilde\nu_j^2-n_j^2)\nu_j'{}^2+(\nu_j^2-\tilde\nu_j^2)n_j'{}^2, \\
(\nu_j^2-n_j^2)\tilde\nu_j\tilde\nu_j' &= (\tilde\nu_j^2-n_j^2)\nu_j\nu_j'+(\nu_j^2-\tilde\nu_j^2)n_jn_j'.
\end{align*}
Eliminating further $\tilde\nu_j'$ by multiplying the first equation with $\tilde\nu_j$ and subtracting the squared second equation, we find after some algebraic manipulations
\[ (\tilde\nu_j^2-n_j^2)(\nu_j^2-\tilde\nu_j^2)(\nu_j'n_j-\nu_jn_j')^2 = 0. \]

Evaluating this at the value $\omega_1$, we see that the last factor is by assumption \eqref{OCE_Random_Assum_eq} not zero. Thus, there are only two cases.
\begin{enumerate}
\item
Either we have $\tilde\nu_j(\omega_1)=\nu_j(\omega_1)\ne n_j(\omega_1)$ and therefore by \eqref{OCE_Random_third_eq} that $\tilde\rho_j=\rho_j$; then we get with \eqref{OCE_Random_third_eq} and \eqref{OCE_Random_first_eq} that $\tilde\nu_j=\nu_j$ and $\tilde\nu_j'=\nu'_j$ holds on the whole set $I$, which means that we have shown $\tilde S_j=S_j$.
\item
Or we have that $\tilde\nu_j(\omega_1)=n_j(\omega_1)$. Then, \eqref{OCE_Random_third_eq} tells us that also $\nu_j(\omega_1)=n_j(\omega_1)$ and thus, by combining \eqref{OCE_Random_first_eq} and \eqref{OCE_Random_second_eq}, that $\tilde\nu_j'(\omega_1)=\nu'_j(\omega_1)$. Furthermore, we know from assumption \eqref{OCE_Random_Assum_eq} that in this case $\nu'_j(\omega_1)\ne n_j'(\omega_1)$ and therefore \eqref{OCE_Random_first_eq} implies $\tilde\rho_j=\rho_j$ from which we again conclude that $\tilde S_j=S_j$.
\end{enumerate}
\end{proof}

As last type of subproblem, we encounter then the interface between the layer $j$ and the layer $j+1$. Similarly to the case of the initial layer, we obtain here from \autoref{Scat_Layered_th} the data
\begin{equation*}
m_j[n_{j+1},n_{j+1}',z_{j+1},z_{j+1}'](\omega,\delta) = \frac{\bar n_{j+1}(\omega, \delta) -\bar n_j(\omega, \delta)}{\bar n_{j+1}(\omega, \delta) +\bar n_j(\omega, \delta)}\e^{-2\i\frac\omega c\bar n_j(\omega,\delta)\bar z_{j+1}(\delta)}.
\end{equation*}

Again, this data allows us to uniquely obtain the variables $n_{j+1}$, $n_{j+1}'$, $z_{j+1}$, and $z_{j+1}'$ from the already reconstructed values $n_j$ and $n_j'$.
\begin{lemma}\label{th_joint_uniqRecLayer}
Let $j\in\{1,\ldots,J-1\}$ be fixed, $I\subset\R$ be an arbitrary subset and $n_j$, $n_j'$ be given such that $n_j(\omega)\ne0$ for every $\omega\in I$ and that there exists a value $\omega_0\in I\setminus\{0\}$ with $\Im(n_j'(\omega_0))>0$. Assume that we have $(n_{j+1},n_{j+1}',z_{j+1},z_{j+1}')$ and $(\tilde n_{j+1},\tilde n_{j+1}',\tilde z_{j+1},\tilde z_{j+1}')$ with
\begin{equation}\label{eq_joint_uniqRecLayer_meas}
m_j[n_{j+1},n_{j+1}',z_{j+1},z_{j+1}'](\omega,\delta) = m_j[\tilde n_{j+1},\tilde n_{j+1}',\tilde z_{j+1},\tilde z_{j+1}'](\omega,\delta)
\end{equation}
for all $\omega\in I$ and $\delta\in\R$.

Then, we have
\[ n_{j+1}(\omega)=\tilde n_{j+1}(\omega),\;n_{j+1}'(\omega)=\tilde n_{j+1}'(\omega),\;z_{j+1}=\tilde z_{j+1},\text{ and }z_{j+1}'=\tilde z_{j+1}' \]
for all $\omega\in I$.
\end{lemma}
\begin{proof}
Comparing again the different orders of decay in $\delta$ in the exponents in \eqref{eq_joint_uniqRecLayer_meas}, we require that the coefficients on both sides coincide:
\begin{align*}
2\delta^2\frac\omega c\Im(n_j'(\omega))(z_{j+1}'-\tilde z_{j+1}') &= 0 \text{ and } \\
4\delta\frac\omega c\left(\Im(n_j(\omega))(z_{j+1}'-\tilde z_{j+1}')+\Im(n_j'(\omega))(z_{j+1}-\tilde z_{j+1})\right) &= 0.
\end{align*}
Because of the assumption that $\Im(n_j'(\omega_0))>0$, this is equivalent to
\[ z_{j+1}'=\tilde z_{j+1}' \text{ and }  z_{j+1}=\tilde z_{j+1}. \]

As in the proof of \autoref{th_joint_uniqRecLayer0}, equation \eqref{eq_joint_uniqRecLayer_meas} for $\delta=0$ then gives us  
\[ 2n_j(\omega)(n_{j+1}(\omega)-\tilde n_{j+1}(\omega))=0, \]
resulting in $n_{j+1}(\omega)=\tilde n_{j+1}(\omega)$.

Finally, dividing both sides of \eqref{eq_joint_uniqRecLayer_meas} by the exponential factors (which we already know to be the same), we get a quadratic equation for $\delta$ and equating the first order terms in $\delta$, we obtain
\[ 2n_j(\omega)(n_{j+1}'(\omega)-\tilde n_{j+1}'(\omega))=0, \]
which yields $n_{j+1}'(\omega)=\tilde n_{j+1}'(\omega)$.
\end{proof}

\section{Conclusions}

We have thus shown that by analysing a layered medium endued with independently uniformly distributed scatterers in each layer with optical coherence tomography, we can reduce the inverse problem of reconstructing the electric susceptibility of the medium to subproblems for each layer separately by a layer stripping argument, provided the homogeneous parts between the different regions are not too small.

Then by combining this imaging method with an elastography setup by recording measurements for different compression states (normal to the layered structure), we find out that this allows for the reconstruction of the optical parameters and leads to a unique reconstructability of all the optical parameters: the electric susceptibilities and positions of the layers, the electric susceptibilities of the randomly distributed particles, their density, and the locations of the regions of these particles (at every compression state). Of course, the recovered shifts of the layer boundaries for the different compression states could then be used in a next step to determine elastic parameters of the medium.

\section*{Acknowledgements}

This work was made possible by the greatly appreciated support of the Austrian Science Fund (FWF) via the special research programme SFB F68 ``Tomography Across the Scales'':
Peter Elbau and Leopold Veselka have been supported via the subproject F6804-N36 ``Quantitative Coupled Physics Imaging'', and Leonidas Mindrinos acknowledges support from the subproject F6801-N36.

This is a pre-print of a contribution published in its final authenticated version as \cite{ElbMinVes20}.

\section*{References}
\printbibliography[heading=none]

\end{document}